\documentclass{amsart}

\usepackage{hyperref, color}
\usepackage{amssymb}

\theoremstyle{plain}
\newtheorem{theorem}                {Theorem}      [section]
\newtheorem*{theorem*}                {Theorem \ref{thm:appl}}

\newtheorem{corollary}    [theorem]  {Corollary}

\theoremstyle{definition}
\newtheorem{example}      [theorem]  {Example}
\newtheorem{remark}       [theorem]  {Remark}
\newtheorem{definition}   [theorem]  {Definition}

\numberwithin{equation}{section}

\DeclareMathOperator{\trace}{trace} 
\DeclareMathOperator{\Vol}{Vol} 
\DeclareMathOperator{\Voll}{Vol_{\Sigma}} 

\newcommand{\field}[1]{\mathbb{#1}}
\newcommand{\real}{\field{R}}

\newcommand{\vol}{v_{\Sigma}}

       \newcommand{\De}{\Delta}

\newcommand{\vep}{\varepsilon}

\newcommand{\la} {\lambda}

       \newcommand{\Si}{\Sigma}

\begin{document}

\title[CMC Surfaces and Finite  Total Curvature]
{Constant Mean Curvature Surfaces in $\mathbb{M}^2(c)\times\real$ and Finite Total Curvature}

\author[Batista]{M\'arcio Batista}     
\address{Institute of Mathematics, Federal University of Alagoas, 
Macei\'o,  AL, CEP 57072-970, Brazil}
\email{mhbs@mat.ufal.br}

\author[Cavalcante]{Marcos P. Cavalcante}     
\address{Institute of Mathematics, Federal University of Alagoas, 
Macei\'o,  AL, CEP 57072-970, Brazil}
\email{marcos@pos.mat.ufal.br}

\author[Fetcu]{Dorel Fetcu}
\address{Department of Mathematics and Informatics, Gh. Asachi Technical University
of Ia\c si, Bd. Carol I, no. 11, 700506 Ia\c si, Romania}
\email{dfetcu@math.tuiasi.ro}

\date{July 13, 2015}

\thanks{The first author was supported by FAPEAL/Brazil. 
The second author was supported by CNPq/Brazil. 
The third author was supported by the grant BJT 373672/2013--6 of CNPq, Brazil.}

\begin{abstract} We consider surfaces with parallel mean curvature vector field and finite total curvature in product spaces of type $\mathbb{M}^n(c)\times\mathbb{R}$, where $\mathbb{M}^n(c)$ is a space form, and characterize certain of these surfaces. When $n=2$, our results are similar to those obtained in \cite{bds} for surfaces with constant mean curvature in space forms.
\end{abstract}

\maketitle

\section{Introduction}

A classical result in the theory of constant mean curvature surfaces (cmc surfaces), proved by H.~Hopf \cite{H}, 
says that round spheres are the only cmc spheres in Euclidean space. Hopf's result was extended, by S.-S.~Chern \cite{C}, to spheres immersed in $3$-dimensional space forms. The key ingredient in both papers \cite{H} and \cite{C} is the 
fact that there exists a quadratic differential form on the surface which is holomorphic when the mean curvature is constant. 
This form is given by the $(2,0)$-part of the quadratic form $\mathcal{Q}$ defined on a cmc surface 
$\Sigma$ in a space form by
$$
\mathcal{Q}(X,Y)=\langle AX,Y\rangle,
$$
for any vector fields $X$ and $Y$ tangent to $\Sigma$, where $A$ is the self adjoint operator associated
to the second fundamental form of $\Sigma$. We note that, if $A_0=A-HI$ is the traceless part of $A$, where $H$ is the mean curvature of the surface
$\Sigma$, then it is easy to see that $A_0$ vanishes, which means that $\Sigma$ is an umbilical surface, 
if and only if the $(2,0)$-part of $\mathcal{Q}$ vanishes.
 
Since the traceless second fundamental form of a surface measures how much 
it deviates from being totally umbilical, it is natural to study the geometry and topology of
complete cmc surfaces with finite total curvature, in the sense that 
the integral of $|A_0|^2$ is finite.
For instance,  P.~B\'erard, M.~do Carmo, and W.~Santos \cite{bds} proved that a cmc surface
in a space form $M^3(c)$, $c\leq 0$  with $H^2>c$ and  finite total curvature  must be compact.
This result was extend to the case of higher codimension in \cite{bs,ccs}.
On the other hand, Ph.\! Castillon \cite{ca} considered cmc hypersurfaces in hyperbolic spaces 
$\mathbb{H}^{n+1}$ with $H^2<1$ and finite total curvature. He showed, among others, that such 
hypersurfaces are diffeomorphic to the interior of a compact manifold. 
The geometry of cmc surfaces in space forms is also related to the finiteness of the index 
of the Jacobi operator as described  in the works of A.~da Silveira \cite{S} and P.~B\'erard, M.~do Carmo, and W.~Santos \cite{bds2}.

The next step in the study of cmc surfaces was to consider isometric immersions on product spaces
of type $M^2(c)\times \real$, where $\mathbb{M}^2(c)$ is a space form of constant curvature $c$. 
In this case,  U.~Abresch and H.~Rosenberg, in their celebrated paper \cite{ar},
introduced a differential form (the Abresch-Rosenberg differential) which is holomorphic on cmc surfaces
and described those cmc surfaces on which this holomorphic differential vanishes.

Taking into account the relation between the quadratic form $\mathcal{Q}$
and the operator $A_0$, it seems natural to ask if an operator similar to $A_0$
can be used to study the geometry of cmc surfaces in product spaces.

In our paper we give an affirmative answer to this question and we actually obtain some more 
general results, for surfaces with parallel mean curvature vector field (pmc surfaces) immersed 
in spaces of type $\mathbb{M}^n(c)\times\mathbb{R}$, where $\mathbb{M}^n(c)$ is a space form. 
Thus, we use the operator $S$ studied in \cite{b,FR} to characterize certain pmc 
surfaces with \emph{finite total curvature}, i.e., such that the integral of $|S|^2$ is finite. 
We pay a special attention to cmc  surfaces in $\mathbb{M}^2(c)\times\mathbb{R}$ and obtain 
similar results to those in \cite{bds} for cmc surfaces in space forms.

Our main result shows that if the norm of the second fundamental form of a complete non-minimal 
pmc surface $\Sigma$ in $\mathbb{M}^n(c)\times\mathbb{R}$ is bounded and the surface has finite total 
curvature, then the function $|S|$ goes to zero uniformly at infinity (Theorem \ref{zero:pmc} and Corollary \ref{zero}). 
As applications, we prove compactness results for pmc surfaces (Theorem \ref{theo-compact0} and Corollaries \ref{theo-compact1}
and \ref{theo-compact2}), as well as lower bounds estimates for the bottom of the essential spectrum of $\Sigma$ (Theorems \ref{thm_spec} and
\ref{index}).

\noindent {\bf Acknowledgements.} The authors would like to thank the referee for very valuable comments and suggestions.

\section{Preliminaries}

Let $\mathbb{M}^n(c)$ be a space form, i.e., a simply connected $n$-dimensional manifold with constant sectional curvature $c$, and consider the product space $\bar M=\mathbb{M}^n(c)\times\mathbb{R}$. Then the curvature tensor $\bar R$ of $\bar M$ is given by
\begin{align}\label{eq:barR}
\bar R(X,Y)Z=&c\{\langle Y, Z\rangle X-\langle X, Z\rangle Y-\langle Y,\xi\rangle\langle Z,\xi\rangle X+\langle X,\xi\rangle\langle Z,\xi\rangle Y\\\nonumber &+\langle X,Z\rangle\langle Y,\xi\rangle\xi-\langle Y,Z\rangle\langle X,\xi\rangle\xi\},
\end{align}
for any tangent vector fields $X$, $Y$ and $Z$, where $\xi$ is the unit vector field tangent to $\mathbb{R}$.

Now, let us consider $\Sigma$ an isometrically immersed surface in $\mathbb{M}^n(c)\times\mathbb{R}$. 
The second fundamental form $\sigma$ of $\Sigma$ is defined by the equation of Gauss
$$
\bar\nabla_XY=\nabla_XY+\sigma(X,Y),
$$
for any vector fields $X$ and $Y$ tangent to the surface, where $\bar\nabla$ and $\nabla$ are the Levi-Civita connections on $\mathbb{M}^n(c)\times\mathbb{R}$ and $\Sigma$, respectively. Then the mean curvature vector field $\vec H$ of $\Sigma$ is given by $\vec H=(1/2)\trace\sigma$. The shape operator $A$ and the normal connection $\nabla^{\perp}$ are defined by the equation of Weingarten
$$
\bar\nabla_XV=-A_VX+\nabla^{\perp}_XV,
$$
for any tangent vector field $X$ and any normal vector field $V$. 

We also have the Gauss equation of the surface $\Sigma$ in $\bar M=M^n(c)\times\mathbb{R}$
\begin{align}\label{Gauss}
\langle R(X,Y)Z,W\rangle=&\langle\bar R(X,Y)Z,W\rangle+\langle\sigma(Y,Z),\sigma(X,W)\rangle\\\nonumber&-\langle\sigma(X,Z),\sigma(Y,W)\rangle,
\end{align}
where $X$, $Y$, $Z$ and $W$ are vector fields tangent to $\Sigma$ and $R$ is the curvature tensor corresponding to $\nabla$.

\begin{definition} If the mean curvature vector field $\vec H$ of a surface $\Sigma$ is
parallel in the normal bundle, i.e., $\nabla^{\perp}\vec H=0$, then $\Sigma$ is called a \textit{pmc surface}. When $n=2$, a pmc surface in $\mathbb{M}^2(c)\times\real$ is just a surface with constant mean curvature $H=|\vec H|$ and it is called 
a \textit{cmc surface}. 
\end{definition}

In \cite{AdCT}, the authors showed that the $(2,0)$-part of the quadratic form $Q$ 
defined on a pmc surface $\Sigma$ immersed in $\mathbb{M}^n(c)\times\mathbb{R}$ by
$$
Q(X,Y)=2\langle\sigma(X,Y),\vec H\rangle-c\langle X,\xi\rangle\langle Y,\xi\rangle
$$ 
is holomorphic. When $n=2$, this is just the Abresch-Rosenberg differential introduced in 
\cite{ar}.  In \cite{b} and then in \cite{FR}, the authors considered an operator $S$ on a non-minimal pmc surface $\Sigma$ given by  
$$
2H\langle SX,Y\rangle=Q(X,Y)-\frac{\trace Q}{2}\langle X,Y\rangle,
$$
or, equivalently,
\begin{equation}\label{s}
SX=\frac{1}{H}A_{\vec H}X-\dfrac{c}{2H}\langle X,T\rangle T+\dfrac{c}{4H}|T|^2X-HX,
\end{equation}
where $X$ and $Y$ are vector fields tangent to $\Sigma$ and $T$ is
the tangential component of the parallel vector field $\xi$. Two remarkable properties of this operator are the facts that $S$ vanishes if and only if the $(2,0)$-part of $Q$ vanishes and also that $S$ satisfies a Simons type equation.

\begin{theorem}[\cite{b,FR}]\label{thm:delta} 
Let $\Sigma$ be an immersed non-minimal pmc surface in  $\mathbb{M}^n(c)\times\real$. 
Then
$$
\frac{1}{2}\Delta|S|^2=2K_{\Sigma}|S|^2+|\nabla S|^2,
$$
where $K_{\Sigma}$ is the Gaussian curvature of the surface.
\end{theorem}

\begin{corollary}\label{simons2} 
Let $\Sigma$ be an immersed non-minimal pmc surface in 
$\mathbb{M}^n(c)\times\mathbb{R}$ such that 
$\mu=\sup_{\Sigma}(|\sigma|^2-(1/H)^2|A_{\vec H}|^2)<+\infty$. Then, the function $u=|S|$ satisfies the following inequality, in the sense of distribution,
$$
-\De u \leq au^3 + bu,
$$
where $a$ and $b$ are constants depending on $c$, $H$, and $\mu$.
\end{corollary}

\begin{proof} 
Let us consider the local orthonormal frame field $\{E_3=\vec H/H,E_4,\ldots,E_{n+1}\}$ in 
the normal bundle, and denote $A_{\alpha}=A_{E_{\alpha}}$.

From the definition \eqref{s} of $S$, we have, after a straightforward computation,
$$
\det A_3=H^2-\frac{1}{2}|S|^2-\frac{c^2}{16H^2}|T|^4-\frac{c}{2H}\langle ST,T\rangle,
$$
and then, using the Gauss equation \eqref{Gauss} and the expression \eqref{eq:barR} of the curvature tensor $\bar R$ of $\mathbb{M}^n(c)\times\mathbb{R}$, the Gaussian curvature $K_{\Sigma}$ of our surface can be written as
$$
K_{\Sigma}=c(1-|T|^2)+H^2-\frac{1}{2}|S|^2-\frac{c^2}{16H^2}|T|^4-\frac{c}{2H}\langle ST,T\rangle+\sum_{\alpha>3}\det A_{\alpha}.
$$

Next, from Theorem \ref{thm:delta}, we obtain
\begin{align}\label{deltaS}
\frac{1}{2}\Delta|S|^2&=
|\nabla S|^2+2\Big(c(1-|T|^2)+H^2-\frac{1}{2}|S|^2-\frac{c^2}{16H^2}|T|^4-\frac{c}{2H}\langle ST,T\rangle\\\nonumber
&+\sum_{\alpha>3}\det A_{\alpha}\Big)|S|^2.
\end{align}

Since $A_{\alpha}$ is traceless for any $\alpha>3$, we have 
$$
-2\sum_{\alpha>3}\det A_{\alpha}=|\sigma|^2-\frac{1}{H^2}|A_{\vec H}|^2\leq \mu.
$$

We note that $|\nabla|S||\leq|\nabla S|$ and, since $S$ is traceless, $|ST|=\dfrac{1}{\sqrt{2}}|T||S|$. 
Then, from \eqref{deltaS}, also using that the Schwarz inequality implies $|\langle ST,T\rangle|\leq|T||ST|$, it is easy to see that
\begin{align*}
\De|S|&\geq -|S|^3 + |S|\left(2c(1-|T|^2) + 2H^2 - \dfrac{c^2}{8H^2} -\mu \right) - \dfrac{|c|}{\sqrt{2}H}|S|^2 \\
&\geq -|S|^3 + |S|\left(2\min\{c,0\} + 2H^2 - \dfrac{c^2}{8H^2}-\mu \right)- \dfrac{|c|}{\sqrt{2}H}|S|^2\\
&\geq -\left(1+\dfrac{|c|}{2\sqrt{2}H}\right)|S|^3 - \left( \dfrac{|c|}{2\sqrt{2}H} - 2H^2 + \dfrac{c^2}{8H^2}-2\min\{c,0\} +{\mu}\right)|S|,
\end{align*}
that completes the proof.
\end{proof}

We end this section by recalling that a pmc surface $\Sigma$ in $\mathbb{M}^n(c)\times\mathbb{R}$ satisfies 
a Sobolev inequality of the form
\begin{equation}\label{sobolev}
\forall f\in C_0^{\infty}(\Sigma),\quad ||f||_2\leq A_{\Sigma}||\nabla f||_1+B_{\Sigma}||f||_1,
\end{equation}
where $||f||_p=(\int_{\Sigma}|f|^p d\vol)^{1/p}$ is the $L^p$-norm of the function $f$ and $A_{\Sigma}$ and $B_{\Sigma}$ are constants that depends only on the mean curvature $H$ of the surface (see \cite{hs}).

\begin{remark}
If $c>0$, we cannot apply the Sobolev inequality of \cite{hs} directly, because 
in this case we have restrictions on the support of the test functions. 
However, it holds (with new constants)  for all smooth functions on  $\Sigma$ with compact support. 
In fact, the second fundamental form of the standard isometric immersion
$i: \mathbb{M}^n(c)\times\mathbb{R} \to \mathbb{R}^{n+2}$ is given by
$$\beta(X,Y)= -\dfrac{1}{\sqrt{c}}(\langle X, Y\rangle -\langle X, \xi\rangle \langle Y,\xi\rangle)\eta,$$
where $\eta$ is chosen as the inward unit normal vector field.

Now, consider the isometric immersions $x: \Sigma \to \mathbb{M}^n(c)\times\mathbb{R}$ and $i\circ x: \Sigma \to\mathbb{R}^{n+2}$. 
A straightforward computation shows that the second fundamental form of the isometric immersion $i\circ x$ is given by
$$\alpha = \sigma +\beta,$$
where $\sigma $ is a second fundamental form of the isometric immersion $x$.

Therefore, the norm of the mean curvature of $\Sigma$ viewed as a surface in $\mathbb{R}^{n+2}$ is bounded above by the norm of the mean curvature of $\Sigma$ in $\mathbb{M}^n(c)\times\mathbb{R}$, which is a constant, plus a positive constant. Thus we can use 
Theorem 2.1 in \cite{hs} to show that the Sobolev inequality (\ref{sobolev}) holds for all $f\in C_0^{\infty}(\Sigma)$.
\end{remark}

%%%%%%%%%%%%%%%%%%%%%%%%%%%%%%%%%%%%%%%%%%%%%%%%
%%%%%%%%%%%%%%%%%%%%%%%%%%%%%%%%%%%%%%%%%%%%%%%%
%%%%%%%%%%%%%%%%%%%%%%%%%%%%%%%%%%%%%%%%%%%%%%%%
\section{The main results and applications}

Let $\Sigma$ be an immersed surface in $\mathbb{M}^n(c)\times\mathbb{R}$ and $x_0\in\Sigma$ be a fixed a point. Consider the Riemannian distance function $d(x_0,x)$ to $x_0$ on $\Sigma$ and the following open domains  
$$
B(R)=\{x\in\Sigma | d(x_0,x)<R\}\quad\textnormal{and}\quad E(R)=\{x\in\Sigma | d(x_0,x)>R\}.
$$

We can now state our main result.

\begin{theorem}\label{zero:pmc}
Let $\Sigma$ be a complete non-minimal pmc surface in $\mathbb{M}^n(c)\times\mathbb{R}$  such that 
%the norm of its second fundamental form $\sigma$ is bounded and 
\begin{equation}\label{eq:finite-curvature}
\int_\Sigma |S|^2 d\vol < +\infty.
\end{equation}
Assume that $\mu=\sup_{\Sigma}(|\sigma|^2-(1/H)^2|A_{\vec H}|^2)<\infty$. 
Then the function $u=|S|$ goes to zero uniformly at infinity. More precisely, there exist positive constants $C_0$ and $C_1$  depending on $c$, $H$ and $\mu$,  and a positive radius $R_\Sigma$ determined by the condition $C_1\displaystyle\int_{E(R_\Sigma)}u^2 d\vol\leq 1$ such that, for all $R\geq  R_\Sigma,$
$$
||u||_{\infty, E(2R)}  \leq C_0 \int_\Sigma u^2 d\vol.
$$
Moreover, there exist some positive constants $D_0$ and $E_0$ depending on $c$, $H$, and $\mu$ such that the inequality
$\displaystyle\int_\Sigma u^2 d\vol \leq D_0$ implies
$$
||u||_{\infty} \leq E_0 \int_\Sigma u^2 d\vol.
$$
\end{theorem}

\begin{proof} Since the Sobolev inequality \eqref{sobolev} is valid to $\Sigma$ and the function $u$ satifies the inequality in Corollary \ref{simons2}, we can work as in the proof of \cite[Theorem 4.1]{bds} to conclude.
\end{proof}

When $n=2$, we have $\mu=0$ and therefore we have following corollary.

\begin{corollary}\label{zero}
Let $\Sigma$ be a complete non-minimal cmc surface in $\mathbb{M}^2(c)\times\mathbb{R}$ with 
$$
\int_\Sigma |S|^2 d\vol < +\infty.
$$ 
Then the function $u=|S|$ goes to zero uniformly at infinity.
More precisely, there exist positive constants $C_0$ and $C_1$  depending on $c$ and $H$, and a positive radius $R_\Sigma$ determined by the condition $C_1\displaystyle\int_{E(R_\Sigma)}u^2 d\vol\leq 1$ such that, for all $R\geq  R_\Sigma,$
$$
||u||_{\infty, E(2R)}  \leq C_0 \int_\Sigma u^2 d\vol.
$$
Moreover, there exist some positive constants $D_0$ and $E_0$ depending on $c$ and $H$ such that the inequality
$\displaystyle\int_\Sigma u^2 d\vol \leq D_0$ implies
$$
||u||_{\infty} \leq E_0 \int_\Sigma u^2 d\vol.
$$
\end{corollary}

In the following we will present some applications of Theorem \ref{zero:pmc}. 
Firstly, we will prove two compactness results for pmc surfaces of finite total curvature.

\begin{theorem}\label{theo-compact0}
Let $\Sigma$ be a complete non-minimal pmc surface in $\mathbb{M}^n(c)\times\real$ with mean curvature vector field $\vec H$ and such that the norm of its second fundamental form $\sigma$ is bounded and 
$$
\int_\Sigma |S|^2d\vol < +\infty.
$$
Then we have
\begin{enumerate}
\item If $c>0$ and $H^2>(\mu+\sqrt{\mu^2+c^2})/4$, then $\Sigma$ is compact$;$

\item If $c<0$ and $H^2>(\mu-2c+\sqrt{\mu^2-4c\mu+5c^2})/4$, then $\Sigma$ is compact,
\end{enumerate}
where $\mu=\sup_\Sigma(|\sigma|^2-(1/H^2)|A_{\vec H}|^2)$. 
\end{theorem}

\begin{proof} As we have seen in the proof of Corollary \ref{simons2}, the Gaussian curvature $K_{\Sigma}$ of $\Sigma$ can be written as
$$
K_{\Sigma}=c(1-|T|^2)+H^2-\frac{1}{2}|S|^2-\frac{c^2}{16H^2}|T|^4-\frac{c}{2H}\langle ST,T\rangle+\sum_{\alpha>3}\det A_{\alpha},
$$ 
where $A_{\alpha}=A_{E_{\alpha}}$, $\{E_3=\vec H/H,E_4,\ldots,E_{n+1}\}$ being a local orthonormal frame field in the normal bundle, and then
\begin{equation}\label{eq:ineqK}
K_{\Sigma}\geq c(1-|T|^2)+H^2-\frac{1}{2}|S|^2-\frac{c^2}{16H^2}-\frac{|c|}{2\sqrt{2}H}|S|+\sum_{\alpha>3}\det A_{\alpha},
\end{equation}
since $|\langle ST,T\rangle|\leq |T||ST|$ and $|ST|=(1/\sqrt{2})|T||S|$.

Next, if $c>0$, since $2\sum_{\alpha>3}\det A_{\alpha}=-(|\sigma|^2-(1/H^2)|A_{\vec H}|^2)\geq-\mu$, we get
$$
K_\Sigma\geq -\frac{1}{2}|S|^2-\frac{c}{2\sqrt{2}H}|S|+H^2-\frac{c^2}{16H^2}-\frac{1}{2}\mu.
$$

When $c<0$, from \eqref{eq:ineqK}, we have
$$
K_\Sigma \geq c-\frac{1}{2}|S|^2-\frac{c}{2\sqrt{2}H}|S|+H^2-\frac{c^2}{16H^2}-\frac{1}{2}\mu.
$$

In both cases, the hypotheses and Theorem \ref{zero:pmc} imply that the negative part $K_\Sigma^-$ of $K_\Sigma$ has compact support and, therefore, satisfies
$$
\int_\Sigma K_\Sigma^-d\vol <+\infty.
$$
It follows, from Huber's Theorem (see \cite[Theorem 1]{w}), that the 
positive part $K_\Sigma^+$ of $K_\Sigma$ also satisfies 
$$
\int_\Sigma K_\Sigma^+d\vol <+\infty.
$$

Next, outside a compact set $\Omega$ we have $K_\Sigma^+\geq k/2>0$, where
$$
k=\begin{cases}
\dfrac{16H^4-8(\mu-2c)H^2-c^2}{16H^2},\quad\textnormal{if}\quad c<0\\
\dfrac{16H^4-8\mu H^2-c^2}{16H^2},\quad\textnormal{if}\quad c>0,
\end{cases}
$$
and then $\Vol(\Sigma\backslash\Omega)<+\infty$. Since the volume of a complete non-compact surface with bounded mean curvture is infinite (see \cite{f}), it follows that our surface $\Sigma$ is compact.
\end{proof}

When $n=2$ and $\mu=0$, we obtain the following result.

\begin{corollary}\label{theo-compact1}
Let $\Sigma$ be a complete cmc surface in $\mathbb{M}^2(c)\times\real$ such that 
$$
\int_\Sigma |S|^2 d\vol < +\infty.
$$
Then we have
\begin{enumerate}
\item If $c>0$ and $H>\sqrt{c}/2$, then $\Sigma$ is compact$;$

\item If $c<0$ and $H>(\sqrt{\sqrt{5}+2}/2)\sqrt{-c}$, then $\Sigma$ is compact.
\end{enumerate}
\end{corollary}

To prove our next result, we will need the following theorem in \cite{FR2}.

\begin{theorem}[\cite{FR2}]\label{deltaT} Let $\Sigma$ be a pmc surface in $\mathbb{M}^n(c)\times\mathbb{R}$. Then we have
$$
\frac{1}{2}\Delta|T|^2=|A_N|^2+c|T|^2(1-|T|^2)-\sum_{\alpha=3}^{n+1}|A_{\alpha}T|^2,
$$
where $N$ is the normal part of the vector field $\xi$, $\{E_{3},\ldots,E_{n+1}\}$ is a local orthonormal frame field in the normal bundle and $A_{\alpha}=A_{E_{\alpha}}$.
\end{theorem}

\begin{theorem}  Let $\Sigma$ be a complete non-minimal pmc surface in $\mathbb{M}^n(c)\times\real$, $c<0$, with mean curvature vector field $\vec H$ and such that the norm of its second fundamental form $\sigma$ is bounded and 
$$
\int_\Sigma(|S|^2+|N|^2)d\vol < +\infty.
$$
If $H^2>(\mu+\sqrt{\mu^2+c^2})/4$, where $\mu=\sup_\Sigma(|\sigma|^2-(1/H^2)|A_{\vec H}|^2)$, then $\Sigma$ is compact.
\end{theorem}

\begin{proof} From Theorem \ref{deltaT}, as $|N|^2=1-|T|^2$, we have
\begin{equation}\label{deltaN}
-\frac{1}{2}\Delta|N|^2=|A_N|^2+c|N|^2(1-|N|^2)-\sum_{\alpha=3}^{n+1}|A_{\alpha}T|^2.
\end{equation}

Next, since $\nabla\xi=0$ implies $\nabla^{\perp}_XN=-\sigma(X,T)$, one obtains
$$
2|N|X(|N|)=X(|N|^2)=2\langle\nabla^{\perp}_XN,N\rangle=-2\langle A_NT,X\rangle,
$$
and then
\begin{equation}\label{nablaN}
|N|^2|\nabla|N||^2=|A_NT|^2.
\end{equation}

Replacing \eqref{nablaN} in \eqref{deltaN}, we get that
$$
-|N|^3\Delta|N|=(|A_N|^2+c|N|^2(1-|N|^2))|N|^2-\sum_{\alpha=3}^{n+1}|A_{\alpha}T|^2|N|^2+|A_NT|^2,
$$
which gives
$$
-|N|^3\Delta|N|\leq(|\sigma|^2+c(1-|N|^2))|N|^4,
$$
since, using the Schwarz inequality, we can see that $\sum_{\alpha=3}^{n+1}|A_{\alpha}T|^2|N|^2\geq|A_NT|^2$. It follows that there exists a constant $d$ such that 
$$
-\Delta|N|\leq -c|N|^3+d|N|.
$$ 

Since the function $w=|N|$ also satisfies the Sobolev inequality \eqref{sobolev} and, by hypothesis, 
$$
\int_{\Sigma}w^2d\vol\leq+\infty,
$$
we can again work as in the proof of \cite[Theorem 4.1]{bds} to prove that also $w$ goes to zero uniformly at infinity.

Now, from \eqref{eq:ineqK}, we get
$$
K_{\Sigma}\geq c|N|^2+H^2-\frac{1}{2}|S|^2-\frac{c^2}{16H^2}-\frac{|c|}{2\sqrt{2}H}|S|-\frac{1}{2}\mu,
$$
which, together with Theorem \ref{zero:pmc}, shows that the superior limit of $K_{\Sigma}$ is positive. This means that we can use the same arguments as in the proof of Theorem \ref{theo-compact0} to conclude.
\end{proof}

\begin{corollary}\label{theo-compact2}  Let $\Sigma$ be a complete non-minimal cmc surface in $\mathbb{M}^2(c)\times\real$, $c<0$, such that
$$
H>\frac{\sqrt{-c}}{2}\quad\mbox{and}\quad\int_\Sigma (|S|^2+|N|^2)d\vol<+\infty,
$$
Then $\Sigma$ is compact.
\end{corollary}
 
In the following, we will find a positive lower bound for the bottom of the essential spectrum of the Laplacian. 
In our theorem, $I_\Sigma$ denotes the index of the Jacobi operator. 

\begin{theorem}\label{thm_spec}
Let $\Sigma$ be a complete non-minimal cmc surface in $\mathbb{M}^2(c)\times\real$, with finite index $I_{\Sigma}$ and 
$$
\int_\Sigma |S|^2 d\vol < + \infty.
$$
Then we have
$$
\la_{ess}^{\De}\geq \begin{cases} 2H^2,\quad\textnormal{if}\quad c>0 \\
\dfrac{(4H^2+c)^2}{8H^2},\quad\textnormal{if}\quad H>\dfrac{\sqrt{-c}}{2}\ \textnormal{and}\  c<0 . 
\end{cases}
$$
\end{theorem}

\begin{proof} Using equation (\ref{s}), we can write the Jacobi operator of $\Sigma$ as
$$
J=\De+\left(|S|^2 + \frac{c}{H}\langle ST ,T \rangle\right) + \frac{1}{8H^2}(c|T|^2+4H^2)^2.
$$

We can see, from Corollary \ref{zero}, that, for any $\vep>0$, there exists $R>0$ such that 
$||S|^2+(c/H)\langle ST ,T \rangle| <\vep$ in $\Si\setminus B(R)$. 
Thus, the index form associated to the Jacobi operator $J$ satisfies
$$
I(f,f)\leq \int_\Si|\nabla f|^2d\vol+ \vep\int_\Si f^2 d\vol-\frac{1}{8H^2}\int_\Si(c|T|^2+4H^2)^2f^2d\vol,
$$
for any $f\in C^\infty_0(\Sigma\setminus B(R))$.

Since the index is finite, we can
choose $R>0$ such that $I$ is nonnegative in $\Sigma\setminus B(R)$ (see \cite{fc}).
Hence we get
$$
\int_\Si |\nabla f|^2 d\vol\geq\frac{1}{8H^2}\int_\Si((c|T|^2+4H^2)^2-\vep)f^2 d\vol , 
$$
for any $f\in C^\infty_0(\Sigma\setminus B(R))$.

Finally, we use the fact that $\la_{ess}^\De = \lim_{R\to\infty} \la_1(\Sigma\setminus B(R))$ to conclude the proof.
\end{proof}

\begin{corollary} Let $\Sigma$ be a surface in $\mathbb{M}^2(c)\times\real$ as in Theorem \ref{thm_spec}. Then there exist two positive numbers $C$ and $a$ such that
$$
\Voll(B(R))\geq C e^{aR},
$$
for any radius $R$ big enough.
\end{corollary}
\begin{proof} The conclusion follows from Theorem \ref{thm_spec} and an improvement of Brook's Theorem
\cite{b1} given in \cite[Theorem 1]{Hi} (see also \cite[Theorem 1.3]{k}).
\end{proof}

Next, we will prove a result about the bottom of the essential spectrum of the
Jacobi operator of $\Sigma$ and a consequence on surface's Morse index.

\begin{theorem}\label{index}  Let $\Sigma$ be a complete non-minimal cmc surface  in 
$\mathbb{M}^2(c)\times\real$ with finite total curvature. Then we have
$$
\la_{ess}^J \leq \la_{ess}^\De - (2H^2-1).
$$
\end{theorem}

\begin{proof} From Corollary \ref{zero}, we have that for any $\vep>0$, there exists $R>0$ such that 
$||S|^2+(c/H)\langle ST ,T \rangle| <\vep$ in $\Si\setminus B(R)$. Then the Jacobi operator $J$ 
of the surface satisfies
\begin{align*}
J&=\De-|T|^2 + |S|^2+\dfrac{c}{H}\langle ST, T\rangle + \dfrac{c^2 |T|^4}{8H^2} + 2H^2 \\
&\geq  \De + 2H^2-|T|^2 -\vep\geq \De +2H^2-1-\vep,
\end{align*}
which leads to the conclusion.
\end{proof}

\begin{corollary}
Let $\Sigma$ be a complete cmc surface  in $\mathbb{M}^2(c)\times\real$ with  $H > \dfrac{1}{\sqrt{2}}$
and finite total curvature. Assume that 
$$
\Voll(B(R))\leq C e^{aR},
$$
for any $R>0$ and some positive constants $C$ and $a<2\sqrt{2H^2-1}$. Then $\Sigma$ has infinite index.
\end{corollary}

\begin{proof} We use Theorem \ref{index} and  \cite[Theorem 3.1]{cz} to conclude 
that $\lambda_{ess}^J < 0$.
\end{proof}

We end by presenting some examples of surfaces that satisfy the hypotheses in Theorem \ref{index}.

\begin{example} In \cite[Theorem3]{ar}  the authors
described three distinct classes of complete non-compact,  possibly immersed,  cmc surfaces
$\Sigma$ in $\mathbb{H}^2\times\real$ such that the quadratic differential $Q^{(2,0)}$, and therefore $S$, vanishes 
on $\Sigma$ and $H<1/2$. More precisely, $\Sigma$ is either
\begin{enumerate}
  \item[$(i)$]  a convex rotationally invariant constant mean curvature graph $D_H^2$ over
    the horizontal leaf $M^2(c)\times\{t_0\}$;

\item[$(ii)$] an embedded annulus, rotationally invariant constant mean curvature surface $C_H^2$
    with two asymptotically conical ends;

\item[$(iii)$] the embedded constant mean curvature surface $P_H^2$; it is an orbit under some
    two dimensional solvable subgroup of ambient isometries. 
\end{enumerate}
\end{example}

\begin{example}
Let  $\Sigma$ be the unit disk $\mathbb{D}\subset \mathbb C$ endowed with a 
complete metric $\lambda^2|dz|^2$. Let $f$ be any holomorphic function on a domain containing  $\mathbb{D}$ in its interior 
and satisfying $|f|\leq \lambda$. 

According to  \cite[Proposition\, 14]{FM} there exists a 2-parameter family of entire $H = 1/2$ graphs in
$\mathbb{H}^2\times\mathbb{R}$  whose Abresch-Rosenberg differential is precisely $4f\, dz^2$
(see also \cite[Theorem 16]{FM1}).

Next, we will verify that such surfaces have finite total curvature. In fact, by the definition of $S$, we have
%$$
%H\langle S(X),Y\rangle = Q(X,Y) - \dfrac{trace( Q)}{2}\langle X, Y\rangle.
%$$
$$
2H\langle SX,Y\rangle=Q(X,Y)-\frac{\trace Q}{2}\langle X,Y\rangle.
$$

Consider $\{\partial_u, \partial_v\}$ an orthonormal frame field associated to some isothermal coordinates and 
let $\partial_z=\partial_u-i\partial_v$. Using this notation, we have
$$4f=Q^{(2,0)} = 2H\langle S\partial_z, \partial_z\rangle = 2H\lambda^2(S_{11} - S_{22} -2iS_{12})=4H\lambda^2(S_{11} - iS_{12}),$$
where we used the fact that $S$ is traceless, and then
$$|f|^2 = H^2\lambda^4(S_{11}^2+S_{12}^2) = H^2\lambda^4|S|^2. $$

Therefore, we get
$$
\int_\Sigma|S|^2\, d\Sigma = \int_\mathbb{D}|S|^2\lambda^2\, dxdy=
\dfrac{1}{H^2}\int_\mathbb{D}|f|^2\lambda^{-2}\, dxdy\leq 
\dfrac{\mbox{Area}(\mathbb{D})}{H^2}=\dfrac{\pi}{H^2},
$$
as claimed. 
\end{example}


\begin{thebibliography}{99}

\bibitem{ar} U.~Abresch, H.~Rosenberg, \textit{A Hopf differential for constant mean curvature
surfaces in $\mathbb{S}^2\times\mathbb{R}$ and $\mathbb{H}^2\times\mathbb{R}$}, Acta Math. 193 (2004), no. 2, 141--174.

\bibitem{AdCT} H.~Alencar, M.~do Carmo, and R.~Tribuzy, {\it A theorem of Hopf and the Cauchy--Riemann inequality},
Comm. Anal. Geom. 15 (2007), no. 2, 283--298.

\bibitem{b} M.~Batista, {\it Simons type equation in $\mathbb{S}^2\times\real$ and $\mathbb{H}^2\times\real$ and applications}, Ann. Inst. Fourier (Grenoble) 61 (2011), no. 4, 1299--1322.


\bibitem{bds} P.~B\'erard, M.~do Carmo, and W. Santos, {\it Complete hypersurfaces with constant mean curvature and finite total curvature}, Ann. Global Anal. Geom. 16 (1998), no. 3, 273--290.

\bibitem{bds2} P.~B\'erard, M.~do Carmo, and W. Santos, 
{\it The index of constant mean curvature surfaces in hyperbolic $3$-space}, 
Math. Z. 224 (1997), no. 2, 313--326.

\bibitem{bs} P.~B\'erard and W. Santos, 
{\it Curvature estimates and stability properties of CMC-submanifolds in space forms}, 
Mat. Contemp. 17 (1999), 77--97.

\bibitem{b1} R.~Brooks, {\it A relation between growth and the spectrum of the Laplacian},
Math. Z. 178 (1981), no. 4, 501--508. 

\bibitem{ccs} M.~do Carmo, L.-F.~ Cheung, and W. Santos
{\it On the compactness of constant mean curvature hypersurfaces with finite total curvature},
Arch. Math. (Basel) 73 (1999), no. 3, 216--222. 

\bibitem{cz} M.~do Carmo, D.  Zhou, {\it Eigenvalue estimate on complete
noncompact Riemannian manifolds and applications}, Trans. Amer. Math. Soc. 351 (1999),
 no. 4, 1391--1401.
 
\bibitem{ca} Ph.~Castillon,
{\it Spectral properties of constant mean curvature submanifolds in hyperbolic space},
Ann. Global Anal. Geom. 17 (1999), no. 6, 563--580.

\bibitem{C} S.-S. Chern, \textit{On surfaces of constant mean curvature in a three-dimensional space of constant curvature}, Geometric dynamics (Rio de Janeiro, 1981), Lecture Notes in Math. 1007, Springer, Berlin, 1983, 104--108.

\bibitem{FM1} I.~ Fern\'andez, P.~ Mira. {\it Harmonic maps and constante mean curvature surfaces in
 $\mathbb H^2\times \real$}. Amer. J. Math. Soc. 129 (2007), no. 4, 1145--1181.

\bibitem{FM} I.~ Fern\'andez, P.~ Mira. {\it Holomorphic quadratic differentials and the Bernstein problem in Heisenberg space}. Trans. Amer. Math. Soc. 361 (2009), no. 11, 5737 - 5752.

\bibitem{FR} D.~Fetcu, H.~Rosenberg, {\it A note on surfaces with parallel mean curvature}, C. R. Math. Acad. Sci. Paris 349 (2011), no. 21--22, 1195--1197.

\bibitem{FR2} D.~Fetcu, H.~Rosenberg, {\it Surfaces with parallel mean curvature in $\mathbb{S}^3\times\mathbb{R}$ and $\mathbb{H}^3\times\mathbb{R}$}, Michigan Math. J. 61 (2012), no. 4, 715--729.

\bibitem{fc}  D.~Fischer-Colbrie, {\it On complete minimal surfaces with finite Morse index
in three-manifolds}, Invent. Math. 82 (1985), no. 1, 121--132.

\bibitem{f}  K.~R.~Frensel, {\it Stable complete surfaces with constant mean curvature}, Bol. Soc. Bras.
Mat. 27 (1996), no. 2, 129--144.

\bibitem{Hi} Y.~Higuchi,  {\it A remark on exponential growth and the spectrum of the Laplacian},
Kodai Math. J. 24 (2001), no. 1, 42--47. 

\bibitem{hs} D.~Hoffman, J.~Spruck, {\it Sobolev and isoperimetric inequalities for Riemannian submanifolds}, Comm. Pure. Appl. Math. 27 (1974), 715--727.

\bibitem{H} H.~Hopf, \textit{Differential geometry in the large}, Lecture Notes in Math. 1000, Springer-Verlag, 1983.

\bibitem{k} H. Kumura, {\it Infimum of the exponential volume growth and the bottom
of the essential spectrum of the Laplacian,}  arXiv:0707.0185v3 (2009).

\bibitem{S} A.~da Silveira,
{\it Stability of complete noncompact surfaces with constant mean curvature},
Math. Ann. 277 (1987), no. 4, 629--638. 

\bibitem{w}  B.~White, {\it Complete surfaces of finite total curvature}, J. Differential Geom. 26 (1987), no. 2, 315--326.

\end{thebibliography}
\end{document}